\documentclass{amsart}
\usepackage{hyperref}
\usepackage{amsthm,amssymb,MnSymbol,epsfig,color}
\usepackage[T1]{fontenc}

\newtheorem{theorem}{Theorem}[section]
\newtheorem{lemma}[theorem]{Lemma}
\newtheorem{corollary}[theorem]{Corollary}
\newtheorem{proposition}[theorem]{Proposition}
\newtheorem{question}[theorem]{Question}

\theoremstyle{definition}
\newtheorem{remark}[theorem]{Remark}
\newtheorem{definition}[theorem]{Definition}

\newcommand{\bZ}{\mathbb{Z}}
\newcommand{\bQ}{\mathbb{Q}}
\newcommand{\bN}{\mathbb{N}}
\newcommand{\bR}{\mathbb{R}}

\newcommand{\cD}{\mathcal{D}}
\newcommand{\cG}{\mathcal{G}}

\newcommand{\mc}[1]{\mathcal{#1}}
\newcommand{\Aut}{\mathrm{Aut}}

\newcommand{\psl}{\mathrm{PSL}}
\newcommand{\pgl}{\mathrm{PGL}}

\newcommand{\Mtwo}[4]{\ensuremath{\begin{pmatrix} {#1} & {#2} \\ {#3} &
    {#4}\end{pmatrix} }}

\newcommand{\RP}{\mathbb{R}\mathrm{P}}
\newcommand{\Cay}{\mathrm{Cay}}
\newcommand{\Leg}[2]{\left(\begin{smallmatrix}{#1}\\-\\{#2}\end{smallmatrix}\right)}
\newcommand{\acts}{\curvearrowright}
\usepackage{amsmath}

\newcommand{\leftQ}[2]{\left.\raisebox{-.2em}{$#2$}\middle\backslash\raisebox{.2em}{$#1$}\right.}

\newcommand{\dotnorm}{\ \dot{\lhd}\ }
\newcommand{\dotsub}{\ \dot{<}\ }

\begin{document}
\author{Alexander Lubotzky}
\address{Institute of Mathematics, Hebrew University
Jerusalem 91904,
Israel.}
\email{alex.lubotzky@mail.huji.ac.il}

\author{Jason Fox Manning}
\address{Department of Mathematics, 310 Malott Hall, Cornell University, Ithaca, NY 14853}
\email{jfmanning@math.cornell.edu}

\author{Henry Wilton}
\address{Centre for Mathematical Sciences, Wilberforce Road, Cambridge, CB3 0WB}
\email{h.wilton@maths.cam.ac.uk}

\title{Generalized triangle groups, expanders,\\ and a problem of Agol and Wise}

\maketitle

\begin{abstract}
  Answering a question asked by Agol and Wise, we show that a desired stronger form of Wise's malnormal special quotient theorem does not hold.  The counterexamples are generalizations of triangle groups, built using the Ramanujan graphs constructed by Lubotzky--Phillips--Sarnak.
\end{abstract}

\section{Introduction}

Consider the following notorious question in geometric group theory (see for example \cite[5.3.B]{Gromov}, \cite[Question 1.15]{bestvina_questions_????}).

\begin{question}\label{q:hyp rf}
Is every hyperbolic group residually finite?
\end{question}

``Dehn filling'' is a powerful technique for constructing hyperbolic groups.  A \emph{group pair} $(G,\mc{P})$ is a group $G$ together with a finite collection $\mc{P}$ of subgroups of $G$.  The subgroups $\mc{P}$ are referred to as \emph{peripheral groups} of the pair.   A \emph{Dehn filling} of a group pair is a quotient
\[
G(N_1,\ldots,N_n):= G/\llangle \cup_i N_i\rrangle
\]
where, for each $i$, the subgroup $N_i$ is normal in $P_i$.  If each $N_i$ is finite index in $P_i$, the filling is said to be \emph{peripherally finite} or \emph{PF}.  We say that a property $\Pi$ holds for \emph{all sufficiently long} Dehn fillings of $(G,\mc{P})$ if there is a finite subset $B\subseteq G\smallsetminus 1$ so that, whenever $N_i\cap B=\varnothing$ for all $i$, the corresponding Dehn filling $G(N_1,\ldots,N_n)$ has $\Pi$.  If all sufficiently long Dehn fillings either satisfy $\Pi$ or are not PF, we say $\Pi$ holds for \emph{all sufficiently long PF Dehn fillings}.
The archetypal Dehn filling theorem is a far-reaching generalization of Thurston's famous Hyperbolic Dehn Filling theorem to the group-theoretic context.
In the context of PF fillings of hyperbolic groups, it has the following consequence (see Definition \ref{def:malnormal} for the definition of almost malnormal).

\begin{theorem}[ \cite{osin:peripheral} (cf.\ \cite{groves_dehn_2006})]\label{thm: GMO}
Let $G$ be hyperbolic, and let $\mc{P} = \{P_1,\ldots,P_n\}$ be an almost malnormal collection of quasiconvex subgroups.  All sufficiently long PF Dehn fillings
\[
G(N_1,\ldots,N_n):= G/\llangle \cup_i N_i\rrangle
\]
are infinite and hyperbolic.
\end{theorem}

Even if one starts with a residually finite hyperbolic group (even a free group) $G$, there is no reason to believe that the resulting Dehn fillings $G(N_1,\ldots,N_n)$ should be residually finite. (Compare \cite[Theorem 8.1]{kapovich_representations_2005}, in which it is shown that an infinite hyperbolic proper quotient $G$ of a linear group $\Gamma$ need not be linear; this applies, for instance, if $G$ is a Dehn filling of $\Gamma$.)   Theorem \ref{thm: GMO} therefore seems like a promising candidate for constructing non-residually-finite hyperbolic groups.

These considerations made the work of Wise and his coauthors on \emph{virtually special groups} all the more surprising.  A nonpositively curved cube complex $X$ is \emph{special} if there is a locally isometric immersion to the Salvetti complex associated to some right-angled Artin group.  A group is \emph{special} if it is the fundamental group of a compact special cube complex.  A group is \emph{virtually special} if it has a special subgroup of finite index.   Virtually special groups have numerous attractive properties.  For example, they are virtually subgroups of right-angled Artin groups, which are linear.  It follows that virtually special groups are linear, and therefore residually finite.  In addition, an infinite virtually special group has a subgroup of finite index with infinite abelianization.

One of the most important theorems about virtually special groups is Wise's Malnormal Special Quotient Theorem, which can be thought of as a Dehn filling result.  In order to state it, we need one additional piece of terminology about Dehn fillings of a group pair $(G,\mathcal{P})$.

We say that a property $\Pi$ holds for \emph{a positive fraction of all} Dehn fillings if, for each $i$, there is a subgroup $\dot{P}_i< P_i$ of finite index so that, whenever $N_i< \dot{P}_i$ for all $i$, the corresponding Dehn filling $G(N_1,\ldots,N_n)$ has $\Pi$.  The Malnormal Special Quotient theorem can now be stated as follows \cite{Wise} (cf.\ \cite{agm:msqt}).

\begin{theorem}[Wise's Malnormal Special Quotient Theorem]\label{thm: MSQT}
 Let $G$ be hyperbolic and virtually special, and let $\mc{P} = \{P_1,\ldots,P_n\}$ be an almost malnormal collection of quasiconvex subgroups.  A positive fraction of all Dehn fillings 
\[
G(N_1,\ldots,N_n):= G/\llangle \cup_i N_i\rrangle
\]
are hyperbolic and virtually special.
\end{theorem}

Thus, remarkably, in the context of virtually special groups, Dehn fillings can be performed that preserve residual finiteness.  This was one of the most important ingredients in Agol's celebrated proof of the Virtual Haken conjecture \cite{agol_virtual_2013}.

Nevertheless, the Malnormal Special Quotient Theorem does not completely rule out the possibility of constructing a non-residually finite hyperbolic group using Dehn filling, since it only applies to a \emph{positive fraction} of all possible Dehn fillings.  As a result of Theorem \ref{thm: GMO}, all sufficiently long PF Dehn fillings of a virtually special group are infinite and hyperbolic, but only a positive fraction of them are guaranteed to be virtually special (and hence residually finite).   One is therefore led to wonder whether the Malnormal Special Quotient Theorem can be given such a form.  This led Ian Agol \cite[Problem 14]{AgolICM} and Daniel Wise \cite[Problem 13.16]{WiseICM} to ask the following question in their 2014 ICM talks.

\begin{question}\label{q:smsqt}
 Let $G$ be hyperbolic and virtually special, and let $\mc{P} = \{P_1,\ldots,P_n\}$ be an almost malnormal collection of quasiconvex subgroups.  Are all sufficiently long PF Dehn fillings
 \[
 G(N_1,\ldots,N_n):= G/\llangle \cup_i N_i\rrangle
 \]
virtually special?
\end{question}
The purpose of the current note is to show that this question has a \emph{negative} answer in some simple situations, meaning that the Malnormal Special Quotient Theorem is in some sense as strong as it can be.   

Our examples will be \emph{$k$--fold triangle groups} (discussed at length in Sections \ref{sec:triangle} and \ref{sec:complexes}).  We briefly give the definition now.
Let $k\geq 2$, and let $G$ be a free product of three copies of $\bZ/k$.\footnote{By $\bZ/k$ we mean the cyclic group of order $k$.}  The collection $\mc{P}=\{P_{12},P_{13},P_{23}\}$ consists of three two-fold free products of copies of $\bZ/k$ obtained each by omitting one of the copies.  Fix a surjection $G\to \bZ/k$ taking each free factor isomorphically to the target, and let $G_0$ be the kernel.  The collection $\mc{P}_0=\{P_{12,0},P_{13,0},P_{23,0}\}$ is the collection of intersections of elements of $\mc{P}$ with $G_0$.  Each $P_{ij,0}$ can be identified with a free group $F_{k-1}$ on $k-1$ generators and with a subgroup of index $k$ in $\bZ/k\ast \bZ/k$.

In Section \ref{sec:triangle}, we define, for normal subgroups of finite index in $P_{ij,0}$ notions of \emph{rotundness} (large girth for some associated graph), and \emph{expansiveness} (good expansion for the associated graph). See Definition \ref{def:rotundexpansive} for the precise definitions.  For the group $G_0$ just described, we prove:
\begin{theorem}\label{thm:thenT}
  If, for each $i \in\{ 1,2,3\}$,  the subgroup $N_{ij,0}\dotnorm P_{ij,0}\cong F_{k-1}$ is rotund and expansive, then $G_0(N_{12,0},N_{13,0},N_{23,0})$ is hyperbolic and has property (T).
\end{theorem}
In fact (see Theorem \ref{thm:rotundandT}) rotundness alone suffices for hyperbolicity; that some lower bound on girth suffices can also be seen from Theorem \ref{thm: GMO}.

In Section \ref{sec:expanders}, we use the Ramanujan graphs constructed in \cite{LPS88} to show the following proposition (note that the group $P_{ij,0}$ is free of rank $k-1$).
\begin{proposition}\label{thm:existexpanders}
  There exists $k\geq 18$ and, for each $1\leq i<j\leq 3$, a sequence $\{K_{ij,n}\}_{n\in \bN}$ of normal, rotund, expansive subgroups of $P_{ij,0}$ so that $\bigcap_{n\in \bN} K_{ij,n} = \{1\}$.
\end{proposition}
Note that each group $P_{ij,0}$ in the statement of the Proposition is free of rank $k-1$, and that the resulting graphs are $k$--valent.  The possible $k$ include $p+1$ for any prime $p\geq 17$ so that $p\equiv 1\pmod{4}$.

\begin{corollary}\label{cor:no}
The answer to Question \ref{q:smsqt} is `no'.
\end{corollary}
\begin{proof}
  Fix a $k$ as in Proposition \ref{thm:existexpanders}.  The pair $(G_0,\mc{P}_0)$ we have just described satisfies:
  \begin{enumerate}
  \item $G_0$ is free, hence hyperbolic.
  \item The elements of $\mc{P}_0$ are quasiconvex.
  \item $\mc{P}_0$ is a malnormal collection.
  \end{enumerate}
  Suppose the answer were `yes'.    Then for some $j$, the quotient $G_0(K_{12,j},K_{13,j},K_{23,j})$ is an infinite virtually special group; in particular it has a finite index subgroup with infinite abelianization \cite{antolin_tits_2015}.  This contradicts property (T).
\end{proof}
It is interesting to point out that the solution to this group theoretic problem relies essentially on number theory, via the construction in \cite{LPS88}.

\begin{question}
Are the examples from Corollary \ref{cor:no} virtually torsion-free?  Residually finite?  Linear?
\end{question}

\subsection{Conventions}
We use the notation $A\dotsub B$ to indicate that $A$ is finite index in $B$, and $A\dotnorm B$ to indicate that $A$ is a finite index normal subgroup of $B$.  If $G$ is a group and $S\subset G$, we use the notation $\llangle S\rrangle$ to denote the normal closure of $S$ in $G$.

\subsection*{Acknowledgments}
The first author was supported by ERC, NSF, and BSF grants.
The second and third authors are grateful to the Mathematical Sciences Research Institute in Berkeley, California, where this project was started during a special semester on Geometric Group Theory in Fall 2016, funded under NSF grant DMS-1440140.  The second author is partially funded by NSF grant DMS-1462263.  The third author is partially funded by EPSRC Standard Grant number EP/L026481/1.  This paper was completed while the third author was participating in the \emph{Non-positive curvature, group actions and cohomology} programme at the Isaac Newton Institute, funded by EPSRC Grant number EP/K032208/1. 

\section{$k$--fold triangle groups}\label{sec:triangle}
In this section, we describe a generalization of the classical triangle groups which we will use to prove the main result of the paper (Corollary \ref{cor:no}).  To motivate, let us recall first the classical (hyperbolic) triangle groups.

Let $l\geq m\geq r \geq 2$ be integers, so that $\frac{1}{l}+\frac{1}{m}+\frac{1}{r}<1$.  Then there exists an essentially unique hyperbolic triangle  with angles $\frac{\pi}{l}$, $\frac{\pi}{m}$, and $\frac{\pi}{r}$.  The group generated by reflections in the sides of this triangle is a cocompact lattice in $SO(2,1)$, with group presentation given by the Poincar\'e polyhedron theorem:
\begin{equation} 
 \Delta(l,m,r) = \langle x_1,x_2,x_3 \mid x_1^2 = x_2^2 = x_3^2 = (x_1x_2)^l = (x_2x_3)^m = (x_1x_3)^r = 1\rangle 
\end{equation}
The group of orientation-preserving elements in $\Delta(l,m,r)$ has index $2$, and the following presentation:
\begin{equation}
  \Delta_0(l,m,r) = \langle a_1,a_2,a_3 \mid a_1^l = a_2^m = a_3^m = a_1a_2a_3 = 1\rangle 
\end{equation}
Both $\Delta(l,m,r)$ and $\Delta_0(l,m,r)$ are often called triangle groups.  Sometimes $\Delta_0$ is called an \emph{ordinary} triangle group or a \emph{von Dyck} group.

Let us propose a generalization of the triangle groups, generated by elements of order $k$ instead of involutions.  We first fix parent groups $G$ and $G_0$, which generalize the orbifold fundamental groups of a mirrored ideal triangle and a pair of pants, respectively.  We also specify some peripheral subgroups.
\begin{definition}[The parent groups]\label{parentgroups}
  Fix $k\geq 2$.  Let $G$ be the free product of three copies of $\bZ/k$,  
  \[ G = \langle x_1,x_2,x_3 \mid x_1^k = x_2^k = x_3^k \rangle. \]
  For $i<j$, let $P_{ij} = \langle x_i,x_j\rangle < G$.
  Let $G_0$ be the kernel of the map $G\to \bZ/k$ taking $x_i$ to $\bar{1}$ for each $i$.  For $i<j$ let $P_{ij,0}= P_{ij}\cap G_0$.
  Then $G_0$ is free of rank $2k - 2$, and each $P_{ij,0}$ is a free factor of rank $k-1$.
\end{definition}

\begin{definition}[Ordinary triangle groups]\label{ordinary}
  Let $L\dotnorm P_{12,0}$, $M \dotnorm P_{13,0}$, and $R \dotnorm P_{23,0}$.  Let $K_0$ be the normal closure of $L\cup M\cup R$ in $G_0$, and define the \emph{(ordinary) $k$--fold triangle group}:
  \[ G_0(L,M,R) = G_0 / K_0. \]
\end{definition}
We will often omit the word `ordinary'.
Notice that the $2$--fold triangle groups are the classical triangle groups, where we take $L = \langle (x_1x_2)^l\rangle$, $M = \langle (x_1x_3)^m\rangle$, and $R = \langle (x_2x_3)^r\rangle$.

Likewise, for $L$, $M$, $R$  normal subgroups in the $P_{ij}<G$, we can define analogues of the triangle groups $\Delta(l,m,r)$.

\begin{definition}[Extended triangle groups]\label{extendedtriangle}
  Suppose that $L\dotnorm P_{12}$, $M\dotnorm P_{13}$, and $R\dotnorm P_{23}$.  Let $K$ be the normal closure of $L\cup M\cup R$ in $G$.  Then we define the \emph{extended $k$--fold triangle group}:
 \[ G(L,M,R) = G/K. \]
\end{definition}

We next note that, if the subgroup $L$, $M$, $R$ are normal subgroups of both the $P_{ij}$ and the $P_{ij,0}$, then the ordinary triangle groups and extended triangle groups are related as one expects.

\begin{lemma}
  Suppose that $L\dotnorm P_{12}$ and $L<P_{12,0}$, that $M\dotnorm P_{13}$ and $M<P_{13,0}$, and that $R\dotnorm P_{23}$ and $R<P_{23,0}$. Then the normal closure of $L\cup M\cup R$ in $G_0$ is equal to $K$.
\end{lemma}
\begin{proof}
  Clearly $K_0 = \llangle L\cup M\cup R \rrangle_{G_0} < K = \llangle L\cup M\cup R \rrangle_G$.  But we can write $K_0 = K_{L,0} K_{M,0} K_{R,0}$, and $K = K_LK_MK_R$ where $K_{L,0} = \llangle L\rrangle_{G_0}$, $K_L = \llangle L\rrangle_G$, and so on.  Thus it suffices to show, for example, that $K_{L,0} = K_L$, that the normal closures of $L$ in $G$ and $G_0$ coincide.  Since $\langle x_1\rangle$ maps onto $G/G_0$, it is enough to show that $x_1^{-1}K_{L,0}x_1 = K_{L,0}$.  
  Consider a generator $z = x^{-1} y x$ of $K_{L,0}$, where $y\in L$, and $x\in G_0$.  Then $x_1^{-1}x x_1\in G_0$, since $G_0\lhd G$, and $x_1^{-1}y x_1 \in L$, since $L\lhd P_{12}= \langle x_1,x_2 \rangle < G$.  Thus $x_1^{-1} z x_1\in K_{L,0}$.  Since $z$ was arbitrary, we have shown that $K_{L,0}$ is normal in $G$, and so equal to $K_L$, as desired.
\end{proof}

\begin{corollary}
If the ordinary and extended triangle groups are both defined, then $G_0(L,M,R)$ is a subgroup of index $k$ in $G(L,M,R)$.
\end{corollary}

\subsection{Malnormality}
\begin{definition}\label{def:malnormal}
  Let $H$ be a group, and $\mc{Q}$ a collection of subgroups of $H$.  Then $\mc{Q}$ is \emph{malnormal} if whenever $h\in H$, $Q,Q'\in \mc{Q}$, and $Q\cap hQ'g^{-1}$ is nontrivial, then $Q=Q'$ and $h\in Q$.

  The collection is \emph{almost malnormal} if whenever $h\in H$, $Q,Q'\in \mc{Q}$, and $Q\cap hQ'h^{-1}$ is infinite, then $Q=Q'$ and $h\in Q$.
\end{definition}

We make the following observation, whose (easy) proof is left to the reader.
\begin{lemma}\label{lem:malnormal}
  With the notation of Definition \ref{parentgroups}, the collection $\mc{P}= \{P_{12},P_{13},P_{23}\}$ is almost malnormal in $G$, and $\mc{P}_0=\{P_{12,0},P_{13,0},P_{23,0}\}$ is malnormal in $G_0$.
\end{lemma}
  Both $G$ and $G_0$ are virtually free, and thus virtually special locally quasiconvex.  In particular, the pairs $(G,\mc{P})$ and $(G_0,\mc{P}_0)$ both satisfy the hypotheses of Theorem \ref{thm: MSQT} and Question \ref{q:smsqt}.

\subsection{Geometric conditions on graphs and triangle groups}
\begin{definition}
  Let $\Gamma$ be a graph.  The \emph{girth} of $\Gamma$ is the length of the shortest circuit in $\Gamma$.
\end{definition}
\begin{definition}\label{def:lambda1}
  If $\Gamma$ is a connected $k$--regular graph, we define the \emph{Laplacian} in terms of the adjacency matrix $A$:
\[ \delta = I - \frac{1}{k} A. \]
  With this normalization, the spectrum of $\delta$ always contains $0$ and lies in the interval $[0,2]$.
  We define $\lambda_1(\Gamma)$ to be the smallest positive eigenvalue of $\delta$.
\end{definition}

For each pair $i<j$, the group $P_{ij}$ acts on the regular $k$--valent tree $T_k$ with quotient equal to a single edge.  For definiteness we fix a planar embedding of this tree, and an oriented edge $e_0$.  Let $x_i$ act on this tree by rotating around $i(e_0)$ and let $x_j$ act by rotating around $t(e_0)$.  In this way we make $T_k$ into a Bass-Serre tree for the $P_{ij}$, considered as a free product.  Any finite index subgroup $N$ of $P_{ij}$ acts on $T_k$ with a finite quotient graph $\leftQ{T_k}{N}$.
\begin{definition}\label{def:rotundexpansive}
  Let $N\ \dotnorm\  P_{ij,0} < \ P_{ij}\cong (\bZ/k)\ast (\bZ/k)$.  We say that $N$ is \emph{rotund} if $\mathrm{girth}(\leftQ{T_k}{N})>6$.  We say that $N$ is \emph{expansive} if $\lambda_1(\leftQ{T_k}{N})>\frac{1}{2}$.
\end{definition}
These characterizations of subgroups as rotund or expansive depend on the particular action of $P_{ij,0}$ on $T_k$ given.
Here is a more precise version of Theorem \ref{thm:thenT}.
\begin{theorem}\label{thm:rotundandT}
  Let $G_0$, $P_{ij,0}$, $L,M,R$ be as in Definition \ref{ordinary}.  
  \begin{enumerate}
  \item If $L,M,R$ are rotund, then $G_0(L,M,R)$ is hyperbolic.
  \item If $L,M,R$ are rotund and expansive, then $G_0(L,M,R)$ has property (T).
  \end{enumerate}
\end{theorem}
Theorem \ref{thm:rotundandT} will be proved in Section \ref{sec:complexes}.  In Section \ref{sec:expanders} we will produce many examples of rotund expansive $L,M,R$, proving Proposition \ref{thm:existexpanders}.

\section{Triangular complexes of groups}\label{sec:complexes}
In this section, we give the geometric framework necessary to understand why the groups discussed in the last section answer Question \ref{q:smsqt}.  In particular, we will prove Theorem \ref{thm:rotundandT}.

The virtually free group $G$ can usefully be thought of as a complex of groups in two ways, both shown in Figure \ref{fig:triangle}. 
\begin{figure}[htbp]
  \centering
  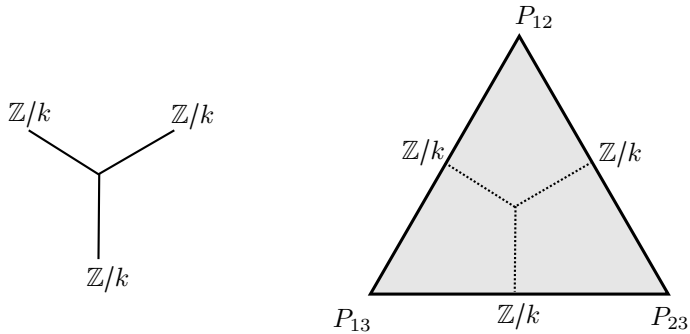
  \caption{$G$ as a graph of finite cyclic groups, and as a triangle of groups.}
  \label{fig:triangle}
\end{figure}
On the left, we see $G$ as the fundamental group of a graph (a tripod) of finite groups $\cG$.  On the right, we see $G$ as a triangle of groups $\cD$, with cyclic edge groups, and vertex groups equal to the peripheral groups.  Both complexes of groups are developable in the sense of \cite[III.$\mc{C}$]{BH}.  This means in the second case that there is an action of $G$ on a simply connected complex (the \emph{development}) with quotient $\cD$, and the complex-of-groups data can be recovered from the action.  Likewise $\cG$ is the quotient of a Bass--Serre tree $T$ by the natural action of $G$.  Here is a way to recover the development in this case: Each $P_{ij}\in \mc{P}$ has a minimal invariant subtree $T_{P_{ij}}$ in $\widetilde{\cG}$. The development $\widetilde{\cD}$ of $\cD$, is homeomorphic to a complex which is obtained from $T$ by coning off each translate of any $T_{P_{ij}}$.  
The link of a vertex in $\widetilde{\cD}$ can be identified with the Bass-Serre tree of $\bZ/k\ast \bZ/k$ (see Lemma \ref{lem2} below).

Likewise, the free group $G_0$ is the fundamental group of a graph $\cG_0$ and a complex of groups $\cD_0$, both shown in Figure \ref{fig:complex}.  Here the vertex groups of $\cD_0$ are the elements of $\mc{P}_0$.  The development of the complex of groups is also $\widetilde{\cD}$.  The link of a vertex of $\cD_0$ is a graph with two vertices joined by $k$ edges.

\begin{figure}[htbp]
  \centering
  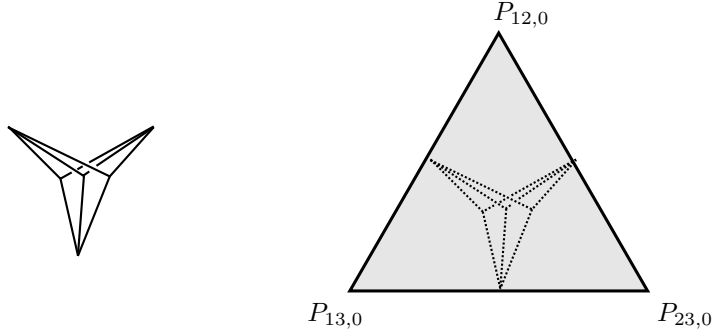
  \caption{$G_0$ is the fundamental group of the graph on the left, and also the fundamental group of the complex of groups with underlying complex $Y$ on the right.  The case $k=3$ is shown.}
  \label{fig:complex}
\end{figure}

Now fix a $k$--fold triangle group $G_0(L,M,R)$ as in Definition \ref{ordinary}.  We obtain a complex of groups structure $\cD_0(L,M,R)$ for $G_0(L,M,R)$ in terms of the one for $G_0$, by replacing the vertex groups (elements of $\mc{P}_0$) with their finite quotients $P_{12,0}/L$, $P_{13,0}/M$, and $P_{23,0}/R$.  

Bridson and Haefliger give a criterion which implies that a given complex of groups is developable.
\begin{theorem}\cite[Theorem III.$\mc{C}$.4.17]{BH}\label{thm:developability}
  If a complex of groups is non-positively curved it is developable.  Moreover, if the local developments are CAT$(-1)$ then the development is CAT$(-1)$.
\end{theorem}
To say that a complex of groups is non-positively curved is precisely to say that the \emph{local developments} are non-positively curved.  This condition depends on how we metrize the cells of the complex.  In our case, we can metrize the triangles as hyperbolic triangles with some constant angle $\theta$.  The local development at a vertex marked by a group $P_{ij,0}/N$ where $N\lhd P_{ij,0}\lhd \bZ/k\ast \bZ/k$ is the hyperbolic cone on the graph $\leftQ{T_k}{N}$, where this graph has been metrized so each edge has length $\theta$.  If $\theta \cdot \mathrm{girth}(\leftQ{T_k}{N})\geq 2\pi$, this local development is locally CAT$(-1)$.  In particular it will satisfy the nonpositive curvature hypothesis in Theorem \ref{thm:developability}.

\begin{proposition}\label{prop:rotund}
  Suppose that $L$, $M$, and $R$ are rotund.  
  \begin{enumerate}
  \item\label{develop} $\cD_0(L,M,R)$ is developable, and the development is contractible;
  \item\label{hyperbolic} $G_0(L,M,R)$ is hyperbolic; and
  \item\label{link} The link of any vertex of the development of $\cD_0(L,M,R)$ is isomorphic to $\leftQ{T_k}{N}$ where $N\in \{L,M,R\}$.
  \end{enumerate}
\end{proposition}
\begin{proof}
  The local development of $\cD_0(L,M,R)$ at a vertex is as described in item \eqref{link}.  Thus if $\cD_0(L,M,R)$ is in fact developable, item \eqref{link} will follow.

  Let $n\geq 7$ be the minimum girth of the graphs $\leftQ{T_k}{N}$ where $N\in \{L,M,R\}$.  Metrizing the triangles of $\cD_0(L,M,R)$ by equilateral hyperbolic triangles with angle $2\pi/n$, we can verify the conditions of Theorem \ref{thm:developability} as discussed above, and see that $\cD_0(L,M,R)$ is developable, establishing the first part of item \eqref{develop}.  The development $X$ is moreover locally (and hence globally) CAT$(-1)$ and thus contractible.  Moreover the group $G_0(L,M,R)$ acts properly cocompactly on $X$.  Thus $G_0(L,M,R)$ is hyperbolic, establishing item \eqref{hyperbolic}.  
\end{proof}

To deduce property (T) when the normal subgroups are expansive, we need the following criterion.
\begin{theorem}\cite[Corollary 1]{BS97}\label{Tcriterion}
  Let $\Gamma\acts Z$ properly and cocompactly, where $Z$ is a contractible simplicial $2$--complex so that for every vertex $v$ of $Z$, the link $Z_v$ of $v$ is connected and satisfies $\lambda_1(Z_v)>\frac{1}{2}$.  Then $\Gamma$ has Property (T).
\end{theorem}

\begin{proposition}\label{thm:expansive}
  If $L$, $M$, and $R$ are rotund and expansive, then $G_0(L,M,R)$ has property (T).
\end{proposition}
\begin{proof}
  Since $L$, $M$, and $R$ are rotund, the group $G_0(L,M,R)$ acts properly and cocompactly on the development $\cD_0(L,M,R)$, which is a contractible complex with each link isomorphic to $\leftQ{T_k}{N}$ for $N\in \{L,M,R\}$.  Since $L,M,R$ are expansive, we have $\lambda_1(\mathrm{link}(v))>\frac{1}{2}$ for each vertex $v$.  We can thus apply Theorem \ref{Tcriterion} to conclude that $G_0(L,M,R)$ has property (T).
\end{proof}

Propositions \ref{prop:rotund} and \ref{thm:expansive} together imply Theorem \ref{thm:rotundandT}.

\section{Finding good expanders}\label{sec:expanders}
In this section we prove Proposition \ref{thm:existexpanders}, about the existence of the expanders we need.  The proposition is phrased in terms of a subgroup $N$, normal and finite index in $P_0 \lhd P \cong \bZ/k\ast \bZ/k$, where $P_0$ is defined to be the kernel of the map $\bZ/k\ast \bZ/k\to \bZ/k$ taking each generator to $\bar{1}\in \bZ/k$.  In applying the results of this section we identify $P_0$ with one of the $P_{ij,0}$ described before.
For $T_k$ equal to the Bass-Serre tree associated to the free splitting of $P$, we are interested in the girth and first eigenvalue of the graphs $\leftQ{T_k}{N}$.  We proceed by identifying $P_0$ with a certain arithmetic subgroup of $\pgl_2(\bQ_p)$.  (We should emphasize that none of the results of the next two subsections are really new, but we want to include enough of the ideas from \cite{LPS88,Lubotzky94} so that the reader gets the flavor of what is going on.)

\subsection{The setup}
If $T$ is a tree, we let $\Aut^+(T)<\Aut(T)$ be the subgroup of index at most two consisting of those $\phi$ which move a point (hence every point) an even distance.  Note that $\Aut^+(T)$ acts on $T$ without inversions, and $\leftQ{T}{\Gamma}$ is bipartite for any $\Gamma<\Aut^+(T)$.  The first lemma is an easy corollary of the fact that $T$ is contractible.
\begin{lemma}\label{lem1}
  Let $T$ be a locally finite tree and $D,\Delta$ two discrete torsion-free subgroups of $\Aut^+(T)$ so that the graphs $\leftQ{T}{D}$ and $\leftQ{T}{\Delta}$ are isomorphic.  Then $D$ and $\Delta$ are conjugate in $\Aut(T)$.
\end{lemma}
The second lemma is also standard.
\begin{lemma}\label{lem2}
  Let $k\in \bZ^+$, and let $P = \bZ/k\ast \bZ/k$, and let $T$ be the Bass-Serre tree for the free splitting of $P$.  (In other words vertices are in one-to-one correspondence with left cosets of the free factors, and if $A_1$ is the first and $A_2$ the second free factor, each $gA_1$ is connected to $gA_2$ by an edge.)
Let $P_0$ be the kernel of the map $P\to \bZ/k$ which is the identity on each free factor.
  \begin{enumerate}
  \item $T$ is a $k$--regular tree.
  \item $\leftQ{T}{P_0}$ consists of $2$ vertices connected by $k$ edges.
  \end{enumerate}
\end{lemma}

 Let $p$ be a prime and  $T_{p+1}$ be the regular $(p+1)$--valent tree, and let $\bQ_p$ be the field of $p$--adic numbers.
 Then $\pgl_2(\bQ_p)$ acts on the Bruhat--Tits tree $T_{p+1}$, as explained in \cite{Serre}.

  Our goal now is to find an \emph{arithmetic} subgroup $\Delta$ of $\psl_2(\bQ_p)<\pgl_2(\bQ_p)<\Aut(T_{p+1})$, so that the quotient $\leftQ{T_{p+1}}{\Delta} \cong \leftQ{T}{P_0}$, the graph from Lemma \ref{lem2}, with $k=p+1$.  Lemma \ref{lem1} implies that $\Delta$ is conjugate to $P_0$ in $\Aut(T_{p+1})$.  

  We will take suitable \emph{congruence} subgroups $\Delta(i)$ of $\Delta$; the isomorphism taking $\Delta$ to $P_0$ will take these congruence subgroups to the groups $L,M,R$ specified in Theorem \ref{thm:rotundandT}.  The fact that the graphs $\leftQ{T}{\Delta(i)}$ are Ramanujan will come from \cite[Section 7.3, Theorem 7.3.12]{Lubotzky94}, using the solution of Deligne to the Ramanujan--Peterson conjecture.  We will see that in fact, one can choose the subgroups $\Delta(i)$ to be nested with $\bigcap_i\Delta(i) = \{1\}$, so that the girth of $\leftQ{T}{\Delta(i)}$ goes to infinity.  

\subsection{Constructing $\Delta$}
Fix $p$ prime, with $p\equiv 1 \pmod{4}$ and $p \geq 17$.

Recall the classical four square theorem of Jacobi \cite[Theorem 386]{HardyWright}: 
\begin{theorem}(Jacobi)
Let $n\in \bZ^+$, and define \[r_4(n) = \#\left\{(x_0,x_1,x_2,x_3)\in \bZ^4\left\mid \sum x_j^2 = n\right.\right\}.\]  Then
\[ r_4(n) = 8\sum_{d\divides n,4\ndivides d} d .\]
\end{theorem}
In particular since $p$ is prime, $r_4(p) = 8(p+1)$.  
We are assuming $p\equiv 1\pmod{4}$, so for any four integers whose squares sum to $p$, exactly three are even.  Thus if we take
\begin{equation} \label{defS}
S = \left\{ (x_0,x_1,x_2,x_3)\in \bZ^4 \mid x_0>0\mbox{ odd, }\sum x_j^2= p \right\},
\end{equation}
then $\# S = p+1$.

Next we claim (again using $p\equiv 1\pmod{4}$) that there exists $\epsilon\in \bZ_p\subseteq \bQ_p$ so that $\epsilon^2 = -1$.  Indeed it is well-known that such an $\epsilon$ exists in $\bZ/p$, and by the Hensel Lemma it can be lifted to $\bZ_p$.
For every $\alpha\in S$, associate the matrix
\begin{equation} \label{defmatrix}
\tilde\alpha = \Mtwo{x_0+\epsilon x_1}{x_2 + \epsilon x_3}{-x_2+\epsilon x_3}{x_0-\epsilon x_1}\in M_2(\bQ_p).
\end{equation}
Note that $\det(\tilde\alpha)= p$; we abuse notation by also thinking of $\tilde\alpha$ as an element of $\pgl_2(\bQ_p)$.  Let $\Gamma$ be the subgroup generated by $\tilde S = \{\tilde\alpha\mid \alpha\in S\}$.

\begin{lemma}
  $\Gamma$ is a discrete cocompact subgroup of $\pgl_2(\bQ_p)$.
\end{lemma}
This is actually a special case of Theorem 7.3.12 of \cite{Lubotzky94}.  Let us explain this special case in some detail.

 Let $H$ be the Hamiltonian quaternion algebra; for a commutative ring $R$, we have
\[ H(R) = \{a_0+a_1 i + a_2 j + a_3 k \mid a_i\in R\},\]
  the associative $R$--algebra generated by symbols $i$, $j$, $k$, satisfying the relations $i^2=j^2=k^2= -1$ and $ij = k = -ji$.  Let $H^*(R)$ be the group of invertible elements in $H(R)$.  

  Since $\bZ[\frac{1}{p}]$ is discrete in $\bR\times \bQ_p$, there are discrete embeddings $H(\bZ[\frac{1}{p}])\hookrightarrow H(\bR)\times H(\bQ_p)$, and 
 $H^*(\bZ[\frac{1}{p}])/\mbox{center}\hookrightarrow (H^*(\bR)/\mbox{center})\times (H^*(\bQ_p)/\mbox{center})$.  Now $H^*(\bR)/\mbox{center} = \RP^3$ is compact, while $H(\bQ_p)\cong M_2(\bQ_p)$ (i.e., $H$ \emph{splits} over $\bQ_p$,  hence $H^*(\bQ_p)/center \cong PGL_2(\bQ_p)$; the map 
\[ a_0+a_1i+a_2j+a_3 k \mapsto \Mtwo{a_0+\epsilon a_1}{a_2 + \epsilon a_3}{-a_2+\epsilon a_3}{a_0-\epsilon a_1} \]
gives the explicit isomorphism).   Since $H^*(\bR)/\mbox{center}$ is compact, the projection of $\Gamma_0 = H^*(\bZ[\frac{1}{p}])/\mbox{center}$ to $H^*(\bQ_p)/\mbox{center} = \pgl_2(\bQ_p)$ gives a discrete subgroup!  (This despite the fact that $\bZ[\frac{1}{p}]$ projected to $\bQ_p$ is dense.)

  Now, our $\Gamma$ is inside the projection, since every $\alpha\in S$ is invertible as an element of $H(\bZ[\frac{1}{p}])$; indeed $\|\alpha\| = \alpha\cdot \bar\alpha = p$ is invertible in $\bZ[\frac{1}{p}]$, and $\alpha^{-1} = \frac{\bar\alpha}{\|\alpha\|}$.  
(In general $\alpha\in H(R)$ is invertible if and only if $\|\alpha\|$ is invertible in the ring $R$.)  One can easily see from \eqref{defS} and \eqref{defmatrix} that $\Gamma \subseteq \Gamma_0(2)$, the mod $2$ \emph{congruence} subgroup  of $\Gamma_0$.  
This all explains why $\Gamma$ is discrete.  But it is also cocompact; in fact $\Gamma = \Gamma_0(2)$, and  every $\tilde\alpha\in \tilde S$ takes the root of the tree (which is the equivalence class of the lattice $\bZ_p^2\subseteq \bQ_p^2$; see \cite{Serre} for this model of the Bruhat-Tits tree) to a sublattice of index $p$ (since $\det(\alpha)=p$) and there are exactly $p+1$ such sublattices -- representing the $p+1$ neighbors of the root vertex.  From this one deduces that $\Gamma$ acts \emph{transitively} on the vertices of the tree $T_{p+1}$. In fact $\Gamma$ acts simply transitively, and is therefore a free group on $\frac{p+1}{2}$ generators. (Note $\tilde{\bar\alpha} = \tilde\alpha^{-1}$ where $\bar\alpha$ is the quaternionic conjugate, and so the image of $S$ is a symmetric subset of $\pgl_2(\bQ_p)$.)   Thus $T_{p+1}$ can be identified with the Cayley graph $\Cay(\Gamma,\tilde S)$. In particular $\leftQ{T_{p+1}}{\Gamma}$ is a bouquet of $k/2$ circles, and hence compact.
\medskip

Now let $\Delta = \Gamma\cap \Aut^+(T_{p+1})$; this is an index-$2$ subgroup of $\Gamma$ which preserves the $2$--coloring of the tree.  Because $\Gamma=\Gamma_0(2)$ is free of rank $\frac{p+1}{2}$, the rank of $\Delta$ is, by the Nielsen--Schreier Theorem, $2(\frac{p+1}{2}-1)+1 = p = k-1$, and there are two orbits of vertices.  In particular, there is an isomorphism $\Psi$  from $\Delta$ to $P_0 = \ker(\bZ/k\ast \bZ/k \to \bZ/k)$ and an equivariant isomorphism from the tree $T_{p+1}$ to the Bass--Serre tree of $\bZ/k\ast \bZ/k$.  In particular, we can find rotund or expansive subgroups of $P_0$ by specifying them in $\Delta$, which we now do.

Let $q\neq p$ be a prime or prime power, so that $\Leg{p}{q} = -1$, i.e., $p$ is not a quadratic residue mod $q$.  As explained in \cite[7.3.12]{Lubotzky94}, in this case $\Gamma_0(2q)$ (the mod $2q$ congruence subgroup) preserves the coloring of the tree, so it lies inside $\Delta$.

Moreover, by \cite[7.3.12]{Lubotzky94}, 
the quotients $\leftQ{T}{\Gamma_0(2q)}$ have the following properties:
\begin{enumerate}
\item They are $k$--regular \emph{Ramanujan} graphs, i.e., $\lambda_1(\leftQ{T}{\Gamma_0(2q)})\geq 1 - \frac{2\sqrt{k-1}}{k}$.
  \item The girth of $\leftQ{T}{\Gamma_0(2q)}$ is at least $\frac{4}{3}\log_p(q)$.
\end{enumerate}
  So for fixed $p$ and $q\to \infty$ we are finished. (Take for example $p=17$, $q = 3^l$ and $l\to\infty$ for a nested family.)

\begin{remark}
  The graphs $\leftQ{T}{\Gamma_0(2q)}$, for $q$ a prime congruent to $1$ mod $4$, are really the same as the Ramanujan graphs presented in Lubotzky--Phillips--Sarnak  \cite{LPS88}.  But one can make use of many other examples, e.g., for $k = p^e+1$, those constructed by Morgenstern \cite{Morgenstern}.
\end{remark}

\subsection{Other constructions.}

Another source of examples is provided by a result communicated to us by Varj\`u \cite{varju_girth_????}.  Let $p$ be a prime, and $k$ an odd divisor of $p-1$ or $p+1$.   Let $K$ be the kernel of a random homomorphism from $P_{ij}\cong \bZ/k\ast \bZ/k$ to $L=SL_2(p)$, and let $\Gamma=\leftQ{T_k}{K}$.

\begin{theorem}[Varj\`u]\label{varju}
There is an absolute constant $c>0$ such that the following holds.  For any odd integer $k$ and for any $\varepsilon>0$, we have
\begin{enumerate}
\item $\mathrm{girth}~\Gamma\geq (1/3-\varepsilon)\frac{\log |\Gamma|}{\log(k-1)}$~, and
\item $\lambda_1(\Gamma)> 1- k^{-c}$~
\end{enumerate}
with probability $1-\varepsilon$, provided $p$ is a sufficiently large prime depending on $k$
and $\varepsilon$ and $k|p-1$ or $k|p+1$.
\end{theorem}

By fixing $\varepsilon<\frac{1}{3}$, taking $k$ large enough that $1-k^{-c}>\frac{1}{2}$, and letting $p$ tend to infinity, we obtain examples which are \emph{extended} $k$--fold triangle groups.  Applying Theorem \ref{thm:rotundandT}, these give many more examples to show the answer to Question \ref{q:smsqt} is `no'.

\begin{remark}
Another construction of negatively curved triangle complexes with prescribed links is provided by \cite[Theorem 2]{BS97}. It is possible that these can also be thought of as Dehn fillings of virtually free groups, in which case these would provide another route to answering Question \ref{q:smsqt}.
\end{remark}

\bibliographystyle{alpha}

\begin{thebibliography}{AGM16}

\bibitem[AGM16]{agm:msqt}
Ian Agol, Daniel Groves, and Jason~Fox Manning.
\newblock An alternate proof of {W}ise's malnormal special quotient theorem.
\newblock {\em Forum of Mathematics, Pi}, 4, 2016.

\bibitem[Ago13]{agol_virtual_2013}
Ian Agol.
\newblock The virtual {H}aken conjecture.
\newblock {\em Doc. Math.}, 18:1045--1087, 2013.
\newblock With an appendix by Agol, Daniel Groves, and Jason Manning.

\bibitem[Ago14]{AgolICM}
Ian Agol.
\newblock Virtual properties of 3-manifolds.
\newblock {\em Proceedings of the 2014 ICM, Volume I}, pages 141--170, 2014.

\bibitem[AM15]{antolin_tits_2015}
Yago Antol\'in and Ashot Minasyan.
\newblock Tits alternatives for graph products.
\newblock {\em J. Reine Angew. Math.}, 704:55--83, 2015.

\bibitem[Bes]{bestvina_questions_????}
Mladen Bestvina.
\newblock Questions in geometric group theory.
\newblock
  \url{http://www.math.utah.edu/~bestvina/eprints/questions-updated.pdf}.

\bibitem[BH99]{BH}
Martin~R. Bridson and Andr{\'e} Haefliger.
\newblock {\em Metric spaces of non-positive curvature}, volume 319 of {\em
  Grundlehren der Mathematischen Wissenschaften [Fundamental Principles of
  Mathematical Sciences]}.
\newblock Springer-Verlag, Berlin, 1999.

\bibitem[B{\'S}97]{BS97}
W.~Ballmann and J.~{\'S}wi{\k{a}}tkowski.
\newblock On {$L^2$}-cohomology and property ({T}) for automorphism groups of
  polyhedral cell complexes.
\newblock {\em Geom. Funct. Anal.}, 7(4):615--645, 1997.

\bibitem[GM08]{groves_dehn_2006}
Daniel Groves and Jason~Fox Manning.
\newblock Dehn filling in relatively hyperbolic groups.
\newblock {\em Israel J. Math.}, 168:317--429, 2008.

\bibitem[Gro87]{Gromov}
Mikhael Gromov.
\newblock Word hyperbolic groups.
\newblock In S.~M. Gersten, editor, {\em Essays in Group Theory}, volume~8 of
  {\em Mathematical Sciences Research Institute Publications}, pages 75--264.
  Springer--Verlag, New York, 1987.

\bibitem[HW79]{HardyWright}
G.~H. Hardy and E.~M. Wright.
\newblock {\em An introduction to the theory of numbers}.
\newblock The Clarendon Press, Oxford University Press, New York, fifth
  edition, 1979.

\bibitem[Kap05]{kapovich_representations_2005}
Michael Kapovich.
\newblock Representations of polygons of finite groups.
\newblock {\em Geom. Topol.}, 9:1915--1951 (electronic), 2005.

\bibitem[LPS88]{LPS88}
A.~Lubotzky, R.~Phillips, and P.~Sarnak.
\newblock Ramanujan graphs.
\newblock {\em Combinatorica}, 8(3):261--277, 1988.

\bibitem[Lub94]{Lubotzky94}
Alexander Lubotzky.
\newblock {\em Discrete groups, expanding graphs and invariant measures},
  volume 125 of {\em Progress in Mathematics}.
\newblock Birkh\"auser Verlag, Basel, 1994.
\newblock With an appendix by Jonathan D. Rogawski.

\bibitem[Mor94]{Morgenstern}
Moshe Morgenstern.
\newblock Existence and explicit constructions of {$q+1$} regular {R}amanujan
  graphs for every prime power {$q$}.
\newblock {\em J. Combin. Theory Ser. B}, 62(1):44--62, 1994.

\bibitem[Osi07]{osin:peripheral}
Denis~V. Osin.
\newblock Peripheral fillings of relatively hyperbolic groups.
\newblock {\em Invent. Math.}, 167(2):295--326, 2007.

\bibitem[Ser80]{Serre}
Jean-Pierre Serre.
\newblock {\em Trees}.
\newblock Springer-Verlag, Berlin-New York, 1980.
\newblock Translated from the French by John Stillwell.

\bibitem[Var]{varju_girth_????}
P\'eter~P. Varj\'u.
\newblock Girth and spectral gap of coset graphs of ${SL_2(F_p)}$.
\newblock In preparation.

\bibitem[Wis]{Wise}
Daniel~T. Wise.
\newblock The structure of groups with a quasiconvex hierarchy.
\newblock To appear in \emph{Annals of Mathematics Studies}, Princeton
  University Press, Princeton, NJ.

\bibitem[Wis14]{WiseICM}
Daniel~T. Wise.
\newblock The cubical route to understanding groups.
\newblock {\em Proceedings of the 2014 ICM, Volume II}, pages 1075--1099, 2014.

\end{thebibliography}

\end{document}